\documentclass[oneside,english]{amsart}

\usepackage[latin1]{inputenc}
\usepackage{amssymb}

\usepackage{amssymb}
\usepackage[]{epsf,epsfig,amsmath,amssymb,amsfonts,latexsym}

\newtheorem{thm}{Theorem}[section]
\newtheorem{lem}[thm]{Lemma}
\newtheorem{prop}[thm]{Proposition}
\newtheorem{cor}[thm]{Corollary}

\theoremstyle{definition}
\newtheorem{defn}[thm]{Definition}
\newtheorem{remark}[thm]{Remark}
\newtheorem{example}[thm]{Example}
\newtheorem{question}[thm]{Question}

\newcommand{\AAA}{\mathcal{A}}
\newcommand{\ZZ}{\mathbb{Z}}

\newcommand{\RR}{\mathbb{R}}

\newcommand{\Aut}{\mathit{Aut}}

\newcommand{\dT}{\rho_\infty}
\newcommand{\act}[2]{{#1} \curvearrowright {#2}}
\newcommand{\Act}[2]{\mathit{Act}({#1},{#2})}
\newcommand{\Alg}[2]{\mathit{Act}_{\mathit{alg}}({#1},{#2})}

\newcommand{\subG}[2]{\mathit{Sub}_{{#1}}({#2})}

\newcommand{\POT}{pseudo-orbit tracing}

\newcommand{\sep}{\mathit{sep}}

\title{Pseudo-Orbit Tracing and Algebraic actions of countable amenable groups}

\author{Tom Meyerovitch}
\address{Tom Meyerovitch\\
Department of Mathematics\\
Ben-Gurion University of the Negev}
\email{mtom@math.bgu.ac.il}
\thanks{The research leading to these results has received funding from the People Programme (Marie Curie Actions) of the European Union's Seventh Framework Programme (FP7/2007-2013) under REA grant agreement no. 333598
 and from the Israel  Science Foundation (grant no. 626/14)}


\begin{document}

\maketitle

\begin{abstract}
Consider a countable amenable group acting by  homeomorphisms on a compact metrizable space.
Chung and Li asked if expansiveness and positive entropy of the action imply existence of an off-diagonal asymptotic pair.  For algebraic actions of polycyclic-by-finite groups, Chung and Li proved it does.
We provide examples showing that Chung and Li's result is near-optimal in the sense that  the conclusion fails for some non-algebraic action generated by a single homeomorphism, and for some algebraic actions of non-finitely generated abelian groups.
On the other hand,
we prove that every expansive action of an amenable group with positive entropy that has  the \POT~   property must admit  off-diagonal
asymptotic pairs.
 Using  Chung and Li's  algebraic characterization of expansiveness,  we prove  the \POT~   property for  a class of expansive algebraic actions.
This class includes every expansive principal algebraic action of an arbitrary countable group.

\end{abstract}

\section{Introduction}
This paper is partly motivated by relatively recent work of Chung and Li \cite{MR3314515}
about the dynamics of countable subgroups of automorphisms for compact groups, and algebraic actions in particular.
Part of the paper \cite{MR3314515} is an investigation of  the relation between topological entropy
 and asymptotic behavior of orbits for actions $\act{\Gamma}{X}$ where $X$ is a compact metrizable group and $\Gamma$ acts by group automorphisms \cite{MR3314515}.
The following question  was posed  and left open by Chung and Li  \cite{MR3314515}:

\begin{question}\label{ques:Chung_Li} 
Let  a countable amenable group $\Gamma$  act by  homeomorphisms  on  a compact
metric space $X$. Suppose the action   $\act{\Gamma}{X}$ is expansive and has positive topological entropy.  Must there be an off-diagonal
asymptotic pair in $X$?
\end{question}
Under the additional assumptions that  $\Gamma$ is  a  polycyclic-by-finite group, that $X$ is a compact abelian  group and  that $\act{\Gamma}{X}$ is an expansive action of $\Gamma$  by automorphisms,  Chung and Li obtained an affirmative answer to Question \ref{ques:Chung_Li} \cite[Theorem $1.2$]{MR3314515}.
On the other hand, as remarked in \cite{MR3314515}, there is a much older example due to Lind and Schmidt    of non-expansive $\mathbb{Z}$-actions (in fact, toral automorphisms) with positive entropy where  all asymptotic pairs are on the diagonal \cite[Expample $3.4$]{MR1678035}.

In addition to expansiveness, the two main assumptions in \cite[Theorem $1.2$]{MR3314515} concern  the algebraic nature of the action (action by automorphisms on  a compact group) and the group theoretic condition on $\Gamma$ (polycyclic-by-finite).
We show that these assumptions are not just artifacts of the proof method. Specifically, we prove:

\begin{thm}\label{thm:alg_exp_pos_ent_no_asymp}
There exists a totally disconnected compact abelian group $X$ and a countable amenable subgroup  $\Gamma \subset \Aut(X)$ so that $\act{\Gamma}{X}$ is expansive, has positive topological entropy and no off-diagonal
asymptotic pairs. In particular,  the group $\Gamma$ can be an abelian group (for instance $\bigoplus_{n \in \mathbb{Z}}(\mathbb{Z}/2\mathbb{Z})$).
\end{thm}

\begin{thm}\label{thm:exapnsive_pos_entropy_no_asymp}
There exists an expansive homeomorphism $T:X \to X$ of a totally disconnected compact metrizable space $X$ with positive topological entropy and no off-diagonal
asymptotic pair.
\end{thm}
Theorem \ref{thm:exapnsive_pos_entropy_no_asymp} might be considered a distant cousin of a striking result due to Ornstein and Weiss  stating that every transformation is bilaterally deterministic \cite{MR0382600}.

In attempt to understand when Chung and Li's question has an affirmative answer, we are led to  explore the \POT~ property. This fundamental dynamical property turns out to  ensure an affirmative answer to Question \ref{ques:Chung_Li}.
The \POT~ property was first introduced and studied by R. Bowen \cite{MR0482842} for $\mathbb{Z}$-actions, motivated by the study of Axiom A maps.
Walters and others continued this study and obtained further consequences of \POT~  \cite{MR518563}.
Chung and Lee recently considered the pseudo-orbit tracing property for actions of (finitely generated) countable amenable groups and showed that topological stability and other important consequences of  the \POT~ property  hold  in this more general setting \cite{1611.08994}.
In relation with Question \ref{ques:Chung_Li} above we have the following:

\begin{thm}\label{prop:finite_type_entropy_off_diagonal}
Let $\Gamma$ be a countable amenable group.  
Every expansive $\Gamma$-action on a compact  metrizable  space that satisfies the \POT~ property and has positive topological entropy admits an off-diagonal  asymptotic pair. 
\end{thm}

The \POT~ property is of interest in the context of algebraic actions.
A particular instance of one of our result applies to \emph{principal algebraic actions}, a well-known class of algebraic actions:
\begin{thm}\label{thm:priniciple_alg_pot}
Let $\Gamma$ be a countable group. Every expansive principal algebraic $\Gamma$-action satisfies the \POT~ property.
\end{thm}

In the paper \cite{MR3314515} Chung and Li also provided an  affirmative answer to Question \ref{ques:Chung_Li} for principal algebraic actions for any countable amenable group $\Gamma$ (in fact they proved much more, see Remark \ref{rem:princ_alg} below).
This is also a  conclusion of Theorem \ref{thm:priniciple_alg_pot} combined with Theorem \ref{prop:finite_type_entropy_off_diagonal}.

The organization of the paper is as follows: In section \ref{sec:pot} we recall the \POT~ property, subshifts and subshifts of finite type in particular. We also prove Theorem \ref{prop:finite_type_entropy_off_diagonal}. In section \ref{sec:pot_alg} we discuss the pseudo-orbit tracing for algebraic actions and derive some consequences, in particular the proof of  Theorem \ref{thm:priniciple_alg_pot}. In section \ref{sec:algebraic_exp_no_asymp}
 we prove Theorem \ref{thm:alg_exp_pos_ent_no_asymp}, providing a negative answer to Question \ref{ques:Chung_Li} within the class of algebraic actions. In section \ref{sec:positive_exp_no_asymp}
we prove Theorem \ref{thm:exapnsive_pos_entropy_no_asymp}.

\textbf{Acknowledgements:} I thank Nishant Chandgotia, Nhan-Phu Chung and  Hanfeng Li for valuable comments on an early version of this paper.

\section{The \POT~ property for actions of countable groups}\label{sec:pot}

Throughout this paper, $\Gamma$ will be a countable group and $X$ will be a compact metrizable space, equipped with a metric $d:X\times X \to \mathbb{R}_+$.
We denote by $\Act{\Gamma}{X}$ the space of all continuous actions  $\act{\Gamma}{X}$. The space $\Act{\Gamma}{X}$ inherits a Polish topology  from $\mathit{Homeo}(X)^\Gamma$. Specifically,  given an enumeration of $\Gamma = \{g_1,g_2 \ldots\}$, we have the  following metric on $\Act{\Gamma}{X}$:
$$ \rho(\alpha,\beta) := \sum_{n=1}^\infty \frac{1}{2^n} \sup_{x \in X}d\left(\alpha_{g_n}(x),\beta_{g_n}(x)\right).$$
In this paper  when the action $\alpha \in \Act{\Gamma}{X}$ is clear from the context we will sometimes  write $g\cdot x$ instead of $\alpha_g(x)$.

\begin{defn}(Compare with \cite[Definition $2.5$]{1611.08994})
Fix $S \subset \Gamma$ and $\delta >0$.
A \emph{$(S,\delta)$ pseudo-orbit}  for $\alpha \in \Act{\Gamma}{X}$ is a $\Gamma$-sequence $(x_g)_{g \in \Gamma}$ in $X$ such that $d(\alpha_s(x_{g}),x_{sg})< \delta$ for all $s \in S$ and $g \in \Gamma$.
We say that a  pseudo-orbit $(x_g)_{g \in \Gamma}$ is \emph{$\epsilon$-traced} by $x \in X$ if $d(\alpha_g(x),x_g) <\epsilon$ for all $g \in \Gamma$.
\end{defn}

\begin{defn}
An action $\alpha \in \Act{\Gamma}{X}$ has the \emph{\POT~ property (abbreviated by p.o.t.)} if for every $\epsilon >0$ there exists $\delta >0$ and a finite set $S \subset \Gamma$ such that every $(S,\delta)$ pseudo-orbit is $\epsilon$-traced by some point $x \in X$.
\end{defn}

\begin{defn}(Compare with \cite[Definition $2.1$]{1611.08994})
We say that $\alpha \in \Act{\Gamma}{X}$ is \emph{topologically stable} if
for every $\epsilon >0$
there exists an open neighborhood $\mathcal{U} \subset \Act{\Gamma}{X}$ of $\alpha$ such that for every $\beta \in \mathcal{U}$ there exists a continuous $f:X \to X$ so that $\alpha_g \circ f= f \circ \beta_g$ for every $g \in \Gamma$ and
$$\sup_{x \in X}d(f(x),x) \le \epsilon.$$
\end{defn}

Chung and Lee \cite[Theorem $2.8$]{1611.08994} proved the following:

\begin{thm}\label{thm:pot_stable}
If an action $\alpha \in \Act{\Gamma}{X}$ is expansive and satisfies p.o.t, then it is topologically stable.
Moreover if $\eta >0$ is an expansive constant for $\alpha$ then:
\begin{enumerate}
\item  For every $0 < \epsilon < \eta$ there exists an open neighborhood $\alpha \in \mathcal{U} \subset \Act{\Gamma}{X}$ so that for every $\beta \in \mathcal{U}$ there exists a unique map $f:X \to X$ so that $\alpha_g \circ f = f \circ \beta_g$ for every $g \in \Gamma$ and $\sup_{x \in X}d(f(x),x) \le \epsilon$.
\item If furthermore $\beta$ as above is expansive with expansive constant  $2\epsilon$, then the conjugating map $f$ above is injective.
\end{enumerate}
\end{thm}

Here is a quick sketch of  proof for Theorem \ref{thm:pot_stable}: If $\beta \in \Act{\Gamma}{X}$ is sufficiently close to $\alpha$ then every $\beta$-orbit is an $(S,\delta)$ pseudo-orbit for $\alpha$. By p.o.t this pseudo-orbit is  $\epsilon$-traced by the $\alpha$-orbit of some point $y$. If $\epsilon$ is sufficiently small, expansiveness of $\alpha$ implies that $y$ as above is unique.  Furthermore, expansiveness implies that the function $f:X \to X$ sending a point $x \in X$ to the unique point $y$ whose $\alpha$-orbit traces the $\beta$-orbit of $x$ is continuous.
Uniqueness of the map $f$  implies it is $\Gamma$-equivariant.
If $2\epsilon$ is an expansive constant for $\beta$, then it is impossible for a single $y$ to $\epsilon$-trace  the  $\beta$-orbits of two distinct  points, so $f$ is injective.

\begin{remark}
Strictly speaking,
 Chung and Lee restricted attention to finitely generated groups in \cite{1611.08994}.
The extension to general countable groups does not require any new ideas.
Note that the following simplification  occurs in the finitely generated case: If $\Gamma$ is generated by a finite set $S$, $\alpha \in \Act{\Gamma}{X}$ satisfies p.o.t if  for every $\epsilon >0$ there exists $\delta >0$   such that every $(S,\delta)$ pseudo-orbit is $\epsilon$-traced by some point $x \in X$.
\end{remark}

\begin{defn}
Given $\alpha \in \Act{\Gamma}{X}$,
 $(x,y) \in X \times X$ is called a \emph{$\alpha$-asymptotic pair} if $\lim_{\Gamma \ni g \to \infty}d(\alpha_g(x),\alpha_g(y))=0$. In other words, if for every $\epsilon>0$ there are at most finitely many  $g \in \Gamma$ so that $d(\alpha_g(x),\alpha_g(y)) >  \epsilon$.
\end{defn}

\begin{proof}[Proof of Theorem \ref{prop:finite_type_entropy_off_diagonal}]
Suppose  $\alpha \in \Act{\Gamma}{X}$  is expansive, satisfies p.o.t and has positive topological entropy.

Fix  $\epsilon >0$ so that $2\epsilon$ is an expansive constant for $\alpha$. Because $\alpha$ satisfies p.o.t there exists $\delta >0$ and a finite set $S \subset \Gamma$ so that every $(\delta,S)$ pseudo-orbit is $\epsilon/2$-traced by some $x \in X$. By further increasing $S$ we can safely assume that $S=S^{-1}$ and that $S$ contains the identity.

Let
$(F_n)_{n=1}^\infty$ be a left-F{\o}lner sequence in $\Gamma$. Denote
$$\partial_S F_n:= \left\{ g \in \Gamma~:~ Sg \cap F_n \ne \emptyset \mbox{ and } Sg \cap F_n^c \ne \emptyset\right\}.$$
Because $(F_n)_{n=1}^\infty$ is a left-F{\o}lner sequence in $\Gamma$, it follows that
\begin{equation}\label{eq:partial_F_n}
\lim_{n \to \infty}\frac{|\partial_S F_n|}{|F_n|} = 0.
\end{equation}

Given a finite $F \subset \Gamma$ and $\delta >0$, a set $Y \subset X$ is  \emph{$(F,\delta)$-separated} if
$$\max_{g\in F}d(\alpha_g(x),\alpha_g(y)) \ge \delta \mbox{ for every distinct } x,y \in Y.$$
Let $\sep_{\delta,F}(X,d)$ denote the maximal cardinality of an $(F,\delta)$-separated set in $X$.

Standard argument give that for every finite $F \subset \Gamma$ and every $\delta>0$ the following holds:
$$ \log \sep_{\delta,F}(X,d) \le |F| \log \sep_{\delta/2,\{1\}}(X,d).$$

Thus by \eqref{eq:partial_F_n}:
\begin{equation}\label{eq:sep_bd_small_o}
\log \sep_{\delta,\partial_S F_n}(X,d) = o(|F_n|),
\end{equation}

For every $n >0$, let $X_n \subset X$ be an $(2\epsilon,F_n)$-separated set of  maximal cardinality.
Because $2\epsilon$ is an expansive constant for $\alpha$, the topological entropy of $\alpha$ is equal to
\begin{equation}\label{eq:h_sep_pos}
h(\alpha) =\lim_{n \to \infty}\frac{1}{|F_n|}\log|X_n| >0.
\end{equation}

By \eqref{eq:sep_bd_small_o} and \eqref{eq:h_sep_pos},
$$  \sep_{\delta,\partial_S F_n}(X,d) = o(|X_n|).$$
In particular, for large enough $n$ there exists distinct $x,x' \in X_n$  so that
\begin{equation}
\max_{g \in \partial_S F_n}d(\alpha_g(x),\alpha_g(x')) < \delta.
\end{equation}

Define $(y_g)_{g \in \Gamma} \in X^\Gamma$ as follows:

\begin{equation}
y_g = \begin{cases}
\alpha_g(x') & g \in F_n\\
\alpha_g(x) &  g \in \Gamma \setminus F_n
\end{cases}
\end{equation}

Then $(y_g)_{g \in \Gamma}$ is a $(S,\delta)$ pseudo-orbit for $\alpha$, and by p.o.t it is $\epsilon$-traced by some $y \in X$.
Thus
$d(\alpha_g(x),\alpha_g(y)) < \epsilon$ for every $g \in \Gamma \setminus F_n$.
This implies that $(x,y)$ is an $\alpha$-asymptotic pair.
Also
$$\max_{F \in F_n}d(\alpha_g(x'),\alpha_g(y)) < \epsilon,$$
Because $\{x,x'\}$ are $(F_n,2\epsilon)$-separated, it follows that
there exists $g \in F_n$ so that $d(\alpha_g(x'),\alpha_g(x)) > 2\epsilon$.
By the triangle inequality, there exists $g \in \Gamma$ so that $d(\alpha_g(x),\alpha_g(y))>\epsilon$, and in particular $x \ne y$.
\end{proof}



We now recall a class of $\Gamma$-actions called  \emph{$\Gamma$-subshifts}. These are also referred to as \emph{symbolic dynamical systems}.

\begin{defn}
Let $\AAA$ be a discrete finite set. Consider $\AAA^\Gamma$ as a (metrizable) topological space with the product topology.
The (left) shift action $\sigma \in \Act{\Gamma}{\AAA^\Gamma}$ is given by: %
\begin{equation}
\sigma_g \cdot (x)_h = x_{g^{-1}h}.
\end{equation}


The pair $(\AAA^\Gamma,\sigma)$ is called  the \emph{full shift} with alphabet $\AAA$ over the group $\Gamma$.
A \emph{$\Gamma$-subshift} is a subsystem of a full shift. In other words, the dynamical systems $(X,\sigma_X)$ where action $\sigma_X := \sigma \mid_X \in \Act{\Gamma}{X}$,
 where $X \subset \AAA^\Gamma$ is closed $\sigma$-invariant subset of $\AAA^\Gamma$.
\end{defn}

Evidently, every $\Gamma$-subshift is expansive. It is also well known that  every expansive action $\act{\Gamma}{X}$ on a totally disconnected compact metrizable space $X$ is isomorphic to a $\Gamma$-subshift. An action  $\act{\Gamma}{X}$ that is isomorphic to a $\Gamma$-subshift is sometimes called a \emph{symbolic dynamical system}. From this abstract point of view, symbolic dynamics is the study of expansive actions on a totally disconnected compact metrizable space.

 Let us recall a class of systems called \emph{subshifts of finite type},
arguably the most important class of systems in symbolic dynamics.  
\begin{defn}\label{def:SFT}
 A \emph{$\Gamma$-subshift of finite type} (\emph{$\Gamma$-SFT}) is a subshift $(X,\sigma_x)$ of the  form: 
\begin{equation}\label{eq:SFT}
X = \left\{x \in \AAA^\Gamma~:~ (g \cdot x)\mid_F \in L \mbox{ for every } g \in \Gamma \right\},
\end{equation}
where  $F \subset \Gamma$ is a finite set and $L \subset \AAA^F$.
\end{defn}

Subshifts of finite type over a countable group $\Gamma$ can be characterized as the set of  expansive $\alpha \in \Act{\Gamma}{X}$   that satisfy p.o.t, where $X$ is a  totally disconnected compact metrizable space \cite{1611.08994,MR2353915,MR518563}.
Another characterization of  subshifts of finite type is given  in terms  of a certain descending chain condition that is reminiscent of the definition of Noetherian rings \cite{MR3493309}.
A variant of this descending chain condition appeared back in the early work of Kitchens and Schmidt on automorphisms of compact groups \cite{MR1036904}.
\begin{remark}
Theorem \ref{prop:finite_type_entropy_off_diagonal} is a direct extension of the classical observation that every subshift of finite type with positive entropy has an off-diagonal  asymptotic pair \cite[Proposition $2.1$]{MR1359979}.
\end{remark}
As we will now see,
in many respects, expansive actions on a compact metrizable space $X$ that satisfy p.o.t can be thought of as ``systems of finite type'', even when $X$ is not totally disconnected.

\begin{defn}
Suppose $X$ is a compact metric space and  $\alpha \in \Act{\Gamma}{X}$. Let
\begin{equation}\label{eq:subG}
\subG{\alpha}{X} := \left\{ Y \subset X ~:~ Y \mbox{ is closed and } \alpha_g(Y) =  Y \mbox{ for every } g \in \Gamma \right\}.
\end{equation}
The Hausdorff metric on the closed subsets of $X$ induces a topology on  $\subG{\alpha}{X}$ that makes it a compact metric space.
We say that $Y \in \subG{\alpha}{X}$ is a \emph{local maximum} if there exists an open neighborhood $\mathcal{U} \subset \subG{\alpha}{X}$  of $Y$ such that $Z \subseteq Y$ for every $Z \in \mathcal{U}$.
\end{defn}

\begin{prop}\label{prop:finite_type}
Let $\alpha \in \Act{\Gamma}{X}$ be expansive and satisfy p.o.t.
If $\beta \in \Act{\Gamma}{Y}$ is expansive and   $\Phi:X \to Y$ is continuous, $\Gamma$-equivariant and  injective then $\Phi(X)\in \subG{\beta}{Y}$ is a local maximum in $\subG{\beta}{Y}$.
\end{prop}

\begin{proof}
Let $\alpha \in \Act{\Gamma}{X}$, $\beta \in \Act{\Gamma}{Y}$ and $\Phi:X \to Y$ be as above.
Then $\beta\mid_{\Phi(X)} \in \Act{\Gamma}{\Phi(X)}$ is isomorphic to $\alpha$ and in particular satisfies p.o.t.
Let $\epsilon$ be an expansive constant for $\beta$. Choose  a finite $S \subset \Gamma$ and $\delta  \in (0,\epsilon)$ so that every $(S,\delta)$ pseudo-orbit for $\beta\mid_{\Phi(X)}$ is $\epsilon/2$-traced by some $y \in \Phi(X)$.
Assume with out loss of generality that $ 1 \in S$.
Now let
\begin{equation}
\mathcal{U}:= \left\{Z \in \subG{\beta}{Y}:~ \sup_{z \in Z}\inf_{ x \in X}\max_{g \in S}d(\beta_{g}(z),\beta_g(\Phi(x))) < \delta/2 \right\}.
\end{equation}

Then $\mathcal{U}$ is an open neighborhood of $\Phi(X)$ in $\subG{\beta}{Y}$.
Now if $Z \in \mathcal{U}$, then by definition of $\mathcal{U}$,
for every $z \in Z$ there is $(y_g)_{g \in \Gamma} \in \Phi(X)^\Gamma$  so that
$$d(\beta_{hg}(z),\beta_{h}(y_g)) < \delta/2 \mbox{ for every } g \in \Gamma,~ h \in S.$$
It follows that  $d(\beta_s(y_{g}),y_{sg})< \delta$ for all $s \in S$ and $g \in \Gamma$, so $(y_g)_{g \in \Gamma}$ is an $(S,\delta)$ pseudo-orbit for $\beta\mid_{\Phi(X)}$.
Thus it is $\epsilon/2$ traced by some  $y \in \Phi(X)$. But for every $g \in \Gamma$ we have
$$d(\beta_g(y),\beta_g(z))< d(\beta_g(y),y_g)+ d(y_g,\beta_g(z)) < \epsilon.$$
It follows that $z = y$, so $z \in \Phi(X)$. It follows that $Z \subseteq \Phi(X)$.

\end{proof}

\begin{remark}\label{rem:stable_intersection}
A subsystem $Z \in \subG{\alpha}{Y}$ is a local maximum if and only if  it satisfies the following \emph{stable intersection property
}: For every decreasing sequence $(Y_n)_{n=1}^\infty \in \subG{\alpha}{Y}^{\mathbb{N}}$
\begin{equation}\label{eq:dec_chain}
Y \supseteq Y_1 \ldots \supseteq Y_n \supseteq Y_{n+1} \supseteq \ldots
\end{equation}
such that $Z = \bigcap_{n=1}^\infty Y_n$,
 there exists $N \in\mathbb{N}$ so that $Y_{n}=Y_N$ for all $n \ge N$.
The equivalence of these two conditions follows because
the relation $\subseteq$ is closed in $\subG{\alpha}{Y} \times \subG{\alpha}{Y}$.
The observation that the stable intersection property characterizes subshifts of finite type is due to K. Schmidt \cite{MR3493309}.
\end{remark}

\section{Pseudo-orbit tracing  for  algebraic actions}\label{sec:pot_alg}
In this section we discuss the pseudo-orbit tracing for algebraic actions and some consequences.
We derive Theorem \ref{thm:priniciple_alg_pot}  as a particular case of  Theorem \ref{thm:X_A_SFT} below, thus establishing \POT~ for a class of algebraic actions.

Let $X$ be a  compact metrizable  \emph{abelian} group.
We  denote by $\Alg{\Gamma}{X} \subset \Act{\Gamma}{X}$  the collection of  $\Gamma$-actions on $X$  by continuous group automorphisms.
Every  $\alpha \in \Alg{\Gamma}{X}$ is  called an \emph{algebraic action}.
We recall some notation, essentially following \cite{MR3314515}:

Denote by $\mathbb{Z}\Gamma$  the group ring of $\Gamma$. Denote by $\ell^\infty(\Gamma)$ the Banach space of all bounded $\mathbb{R}$-valued functions
on $\Gamma$, equipped with the $\|\cdot\|_\infty$-norm. Also,  denote  by $\ell^1(\Gamma)$  the Banach
algebra of all absolutely summable $\mathbb{R}$-valued functions on $\Gamma$, equipped with the
$\ell^1$-norm $\|\cdot\|_1$ and the involution $f \mapsto f^*$ defined by $\left(\sum_{s \in \Gamma}f_ss \right)^* := \sum_{s \in \Gamma}f_s s^{-1}$.

For $k \in \mathbb{N}$ and $p \in[1,\infty]$, we write $\ell^p(\Gamma,\mathbb{R}^k):= (\ell^p(\Gamma))^k$, equipped with the suitable $\|\cdot\|_p$-norm.
We denote by $\ell^p(\Gamma,\ZZ^k)$ the integer valued elements of $\ell^p(\Gamma,\mathbb{R}^k)$.
For $k \in \mathbb{N}$, let $M_k(\ell^1(\Gamma))$ denote the Banach algebra of $k \times k$ matrices with $\ell^1(\Gamma)$-entries, with the norm
$$ \| (f_{i,j})_{ 1\le i,j \le n}\|_1 := \sum_{1\le i,j \le n} \| f_{i,j}\|_1.$$
The involution on $\ell^1(\Gamma)$ also extends naturally to an isometric linear involution on $M_k(\ell^1(\Gamma))$ given by
$$ (f_{i,j})^*_{1 \le i \le k,\; 1 \le i \le k} :=  (f_{j,i}^*)_{1 \le i \le k,\; 1 \le i \le k}.$$

The following is a classical and crucial fact in the theory of algebraic actions:
Pontryagin
duality yields a natural one-to-one correspondence between algebraic actions of $\Gamma$ and (discrete, countable) $\mathbb{Z}\Gamma$-modules.
Thus, to each  $\mathbb{Z}\Gamma$-module $\mathcal{M}$  corresponds  an algebraic action  $\alpha^{(\mathcal{M})}\in \Alg{\Gamma}{\widehat{\mathcal{M}}}$. The dynamics of an algebraic action $\alpha^{(\mathcal{M})}$ are completely determined  by the algebraic properties of the dual as a $\mathbb{Z}\Gamma$-module.  Over the years, a significant number of important dynamical properties have found beautiful
algebraic interpretations in terms of the dual module. For $f \in \mathbb{Z}\Gamma$, we let $X_f$ denote the dual group of the $\mathbb{Z}\Gamma$-module $\mathbb{Z}\Gamma / \mathbb{Z}\Gamma f$. The corresponding algebraic action $\alpha^{(f)} \in \Alg{\Gamma}{X_f}$ is called the \emph{principal algebraic action } associated with $f$.

\begin{defn}
If $X$ is a compact group with identity element $1 \in X$, and  $\alpha \in \Alg{\Gamma}{X}$, a point $x \in X$ is called \emph{homoclinic} with respect to the action $\alpha$ if $(x,1)$ is an $\alpha$-asymptotic pair. In this case the \emph{homoclinic group}, denoted by $\Delta(X)$, is the set of all homoclinic points in $X$.
\end{defn}
It is straightforward to check that  $\Delta(X)$ is a $\Gamma$-invariant subgroup and $(x,y)$ is a  an $\alpha$-asymptotic pair if and only if $xy^{-1} \in \Delta(X)$.

We set up some more notation:

Let
\begin{equation}
P:(\ell^\infty(\Gamma))^k \to ((\mathbb{R}/\mathbb{Z})^k)^\Gamma
\end{equation}
denote the canonical projection map,
and  let $\dT$ be the metric on $(\mathbb{R} / \mathbb{Z})^k$ given by
\begin{equation}\label{eq_dT}
\dT \left(v+\mathbb{Z}^k ,w+\mathbb{Z}^k\right):=
\min_{m \in \mathbb{Z}^k}\| v-w -m\|_\infty,~  v,w \in \mathbb{R}^k
\end{equation}

\begin{lem}\label{lem:lift_approx}
For  any $\delta < 1/2$ and $\tilde a,\tilde b \in (\RR/\ZZ)^k$ satisfying $\rho_\infty(\tilde a,\tilde b) < \delta$, the following holds: Let   $a \in \RR^k$ be the unique
element of $[-\frac{1}{2},\frac{1}{2})^k$ such that $\tilde a = a + \ZZ^k$. Then there exists a unique $b \in [-1,1]^k$ so that $\tilde b = b + \ZZ^k$ and
$\|a - b\|_{\infty} < \delta$.
\end{lem}
\begin{proof}
Existence and uniqueness of $a$ as above follows from the fact that
$\RR^k= \biguplus_{n \in \ZZ^k}\left([-\frac{1}{2},\frac{1}{2})^k+n\right)$.
Similarly, there exists a unique $b \in [-\frac{1}{2},\frac{1}{2})^k + a$ so that
$b = \tilde b + \ZZ^k$.
Note that $[-\frac{1}{2},\frac{1}{2})^k +  a \subseteq [-1,1]^k$ because $ a \in [-\frac{1}{2},\frac{1}{2})^k$, so $b \in [-1,1]^k$.
From the definition of $\rho_\infty$ in \eqref{eq_dT},
there exists $n \in \ZZ^k$ so that $\|a - (b + n)\|_\infty < \delta$. 
It follows that
$$\|n\|_\infty =\|b -(b - n)\|_\infty < \|b- a\|_\infty + \|a- (b+n)\|_\infty \le \frac{1}{2} + \delta < 1.$$
It follows that  $n=0$, so $\|b-a\|_\infty < \delta$.
\end{proof}

\begin{lem}\label{lem:lift_compt}
Suppose $\delta \in (0,1/2)$ and that $K,W \subset \Gamma$ are finite subsets that contain the identity element of $\Gamma$. 
Assume
$(x^{(g)})_{g \in \Gamma} \in  ((\RR/\ZZ)^k)^\Gamma$ satisfy  
\begin{equation}\label{eq:dT}
\dT (x^{(g)}_{w^{-1}f},x^{(wg)}_f)< \delta \mbox{ for every } g \in \Gamma,~ f \in K  \mbox{ and }w \in K^{-1}W.
\end{equation}


Then there exists $y^{(g)} \in ([-1,1]^k)^\Gamma \subset (\ell^\infty(\Gamma))^k$ with $P(y^{(g)})=x^{(g)}$ for every $g \in \Gamma$ so that:
\begin{equation}\label{eq:ys_compat}
\|y^{(g)}_{h^{-1}f}- y^{(hg)}_f\|_\infty < 2\delta \mbox{ for every } g \in \Gamma,~ f \in K  \mbox{ and }h \in W.
\end{equation}
\end{lem}

\begin{proof}
Let $(x^{(g)})_{g \in \Gamma} \in ( (\RR/\ZZ)^k)^\Gamma$  satisfy \eqref{eq:dT}.
Define $y^{(g)} \in ([-1,1]^k)^\Gamma \subset (\ell^\infty(\Gamma))^k$  as follows:

First, for every $g \in \Gamma$ let $y^{(g)}_1$ be the unique $a \in [-\frac{1}{2},\frac{1}{2})^k$ so that $P(a)= x^{(g)}_1$.
Next, for every $g \in \Gamma$, $h \in W$ and $f \in K \setminus \{h\}$, let $y^{(g)}_{h^{-1}f}$
be the unique element of $[-1,1]^k$ satisfying
$$\|y^{(g)}_{h^{-1}f}- y^{(f^{-1}hg)}_1\|_\infty < \delta.$$
Existence and uniqueness of such elements follow by applying  Lemma \ref{lem:lift_approx} above with $\tilde b =x^{(g)}_{h^{-1}f}$ and $\tilde a =x^{(f^{-1}hg)}_1$, noting that  $\dT(x^{(g)}_{h^{-1}f},x^{(f^{-1}hg)}_1)< \delta$ by \eqref{eq:dT} because $1 \in K$.
Finally, for every $g \in \Gamma$, $h \in \Gamma \setminus W^{-1}K$, let
$y^{(g)}_h$ be the unique $a \in [-\frac{1}{2},\frac{1}{2})^k$ so that $P(a)= x^{(g)}_h$.
It now follows that for every $g \in \Gamma$, $f \in K$ and $h \in W$,
$$\|y^{(g)}_{h^{-1}f}- y^{(hg)}_f\| \le \|y^{(g)}_{h^{-1}f}- y^{(f^{-1}hg)}_1\| + \|y^{(f^{-1}hg)}_1 - y^{(hg)}_f\| < 2\delta.$$
\end{proof}
For $A \in M_k(\mathbb{Z}\Gamma)$, denote $\displaystyle X_A := \widehat{(\mathbb{Z}\Gamma)^k / ( \mathbb{Z}\Gamma)^kA}$.
Explicitly:
\begin{equation}\label{eq:def_X_A}
X_A := \left\{ x \in ((\mathbb{R}/\mathbb{Z})^k)^\Gamma~:~ (xA^*)_g = \mathbb{Z}^k \mbox{ for every } g \in \Gamma\right\}
\end{equation}
Let $\alpha^{(A)} \in \Alg{\Gamma}{X_A}$ denote the canonical  algebraic action on $X_A$.

Theorem \ref{thm:priniciple_alg_pot}  is a particular instance of the  following more general result, inspired by \cite{MR3314515}:
\begin{thm}\label{thm:X_A_SFT}
Let 
$k \in \mathbb{N}$ and $A \in M_k(\mathbb{Z}\Gamma)$ be invertible in $M_k(\ell^1(\Gamma))$. Then the canonical action
$ \alpha_A \in \Alg{\Gamma}{X_A}$ is expansive and satisfies p.o.t.
\end{thm}

\begin{proof}
To avoid some subscripts, in this  proof we let
\begin{equation}
X:= X_A \mbox{ and }
\alpha:= \alpha_A \in \Alg{\Gamma}{X}.
\end{equation}
By \cite[Lemma $3.7$]{MR3314515}, $\alpha \in  \Alg{\Gamma}{X}$ is expansive.
Fix a metric  $d$ on $X$.

Let   $\epsilon >0$ be arbitrary.
 By definition of p.o.t,
we need to show that there exists  $\delta' >0$ and a finite set $W \subset \Gamma$ so that every $(W,\delta')$ pseudo-orbit is $\epsilon$-traced by some point $x \in X$.

Choose $\delta >0$ small enough so that
\begin{equation}\label{eq:epsilon_small_A}
 \delta \le \min\left\{\frac{1}{4}\|A\|_1^{-1},\frac{1}{4}, \epsilon\right\}.
\end{equation}

From the fact that $\alpha$ is expansive, it follows that by possibly making $\delta$ even smaller we have
\begin{equation}\label{eq:delta_exp}
 \sup_{ g \in \Gamma}\dT( x_{g},\tilde x_{g}) < 2\delta  \mbox{ implies } x=\tilde x.
\end{equation}
In fact, the proof of \cite[Lemma $3.7$]{MR3314515} shows that \eqref{eq:epsilon_small_A} already implies \eqref{eq:delta_exp}, but we will not need this.

Since A is invertible in $M_k(\ell^1(\Gamma))$ so is $A^*$. Let $B \in M_k(\ell^1(\Gamma)) \cong \ell^1(\Gamma, M_k)$ denote the inverse of $A^*$.
Using the natural  identification of $M_k(\ell^1(\Gamma))$ as a subset of the $M_k(\mathbb{R})$-valued functions on $\Gamma$, write $B= \sum_{g \in \Gamma} B_g$ with $B_g \in M_k(\mathbb{R})$. 
There exists a  finite symmetric set  $F\subset \Gamma$  containing the identity with the property that
\begin{equation}\label{eq:F_1_approx}
\sum_{g \in \Gamma \setminus F} \|B_g\| < \frac{\delta}{2}\|A^*\|^{-1}_1.
\end{equation}


Choose a finite symmetric set $S \subset \Gamma$  containing the identity that supports  $A^*$ in the sense that there exists $(A^*_s)_{s \in S} \in (M_k(\mathbb{Z}))^S$ so that
$$(yA^*)_g = \sum_{s \in S}y_{gs^{-1}}A^*_s \mbox{ for every } y \in (\ell^\infty(\Gamma))^k \mbox{ and } g \in  \Gamma,$$
and let
\begin{equation}
K := FS^{-1} \subset \Gamma.
\end{equation}

Choose $\delta'>0$ small enough so that
\begin{equation}\label{eq:delta_prime}
d(x,y) < \delta' \mbox{ implies } \max_{f \in K}\dT( x_f,y_f) < \delta.
\end{equation}


By compactness of $X$ and \eqref{eq:delta_exp} it follows that
there exists a finite set $W \subset \Gamma$ with the property that:

\begin{equation}\label{eq:W_delta_prime}
\max_{ g \in W}\dT( x_{g^{-1}},\tilde x_{g^{-1}}) < 2\delta \mbox{ implies } d(x,\tilde x) < \min\{\delta',\epsilon\} \mbox{ for all } x,\tilde x \in X.
\end{equation}

Let $W \subset \Gamma$ be such a finite set, so that in addition $F^{-1} \subset W$.

 Suppose $(x^{(g)})_{g \in \Gamma} \in X^\Gamma$ is a $(K^{-1}W,\delta')$ pseudo-orbit.
Then by  \eqref{eq:delta_prime}, it follows that  \eqref{eq:dT} holds.


Let   $y^{(g)} \in ([-1,1]^k)^\Gamma \subset (\ell^\infty(\Gamma))^k$ be as in the conclusion of  Lemma \ref{lem:lift_compt}.
Let
\begin{equation}
z^{(g)}:= y^{(g)}A^* \in \ell^\infty(\Gamma,\mathbb{Z}^k).
\end{equation}
Then \eqref{eq:ys_compat} together with \eqref{eq:epsilon_small_A} imply
\begin{equation}\label{eq:z_approx}
\|z^{(g)}_{h^{-1}f} - z^{(hg)}_f\|_\infty <1 \mbox{ for every } g \in \Gamma~,~  f \in K \mbox{ and } h \in W.
\end{equation}
But $z^{(g)}_{h^{-1}f} ,z^{(hg)}_f \in\ZZ^k$ so

\begin{equation}\label{eq:z_g_h_f}
z^{(g)}_{h^{-1}f} =  z^{(hg)}_f \mbox{ for every } g \in \Gamma~,~  f \in K \mbox{ and } h \in W.
\end{equation}
Also note that

\begin{equation}\label{eq:z_infty_norm}
\|z^{(g)}\|_\infty \le \|y^{(g)}\|_\infty \|A^*\|_1 \le \|A^*\|_1.
\end{equation}

Define $z \in  \ell^\infty(\Gamma,\mathbb{Z}^k)$ by
\begin{equation}\label{eq:z_def}
z_{g^{-1}} := z^{(g)}_1 \mbox{ for every } g\in \Gamma.
\end{equation}

Using \eqref{eq:z_g_h_f} and the fact that $F^{-1} \subset W$ it follows that
\begin{equation}\label{eq:g_z_F}
 \alpha_g(z)\mid_F = z^{(g)}\mid_F.
\end{equation}
By \eqref{eq:F_1_approx}, \eqref{eq:z_infty_norm} and \eqref{eq:g_z_F}:
\begin{equation}\label{eq:z_B_trace}
 \| (\alpha_g (z) B)_1 - (z^{(g)} B)_1 \| < \delta \mbox{ for every } g \in \Gamma.
\end{equation}
Let
\begin{equation}
y:=  z B \in (\ell^\infty(\Gamma))^k \mbox{ and } x:= P(y).
\end{equation}
It follows that

$$y A^*= z\in  \ell^\infty(G,\mathbb{Z}^k).$$
Thus $x \in X_A$, recalling the definition of $X_A$ in \eqref{eq:def_X_A}.

It follows that from \eqref{eq:z_B_trace} and \eqref{eq:dT} that
$$\rho_{\infty}((\alpha_g(x))_{f^{-1}},x^{(g)}_{f^{-1}}) < 2\delta \mbox{ for every } g \in \Gamma \mbox{ and } f \in W.$$
By  \eqref{eq:W_delta_prime}  this implies
$$d(\alpha_g(x),x ^{(g)}) < \epsilon \mbox{ for every } g \in \Gamma.$$
We found $x \in X_A$  that  $\epsilon$-traces $(x^{(g)})_{g \in \Gamma}$, so the proof is complete.
\end{proof}

Theorem \ref{thm:priniciple_alg_pot} follows from Theorem \ref{thm:X_A_SFT} above by letting $k=1$, so $A \in M_k(\mathbb{Z}\Gamma) \cong \mathbb{Z}\Gamma$.

Combining Theorem \ref{thm:priniciple_alg_pot} with Theorem \ref{prop:finite_type_entropy_off_diagonal} we recover the following result:
\begin{cor}\label{cor:princ_alg_homoc}
Let $\Gamma$ be a countable amenable group. Every expansive principal algebraic $\Gamma$-action with positive topological entropy admits a non-diagonal asymptotic pair.
Equivalently, it has a non-trivial homoclinic group.
\end{cor}

\begin{remark}\label{rem:princ_alg} By \cite[Corollary $7.9$]{MR3314515} every non-trivial expansive principal algebraic action of an amenable group has positive entropy.
By  \cite[Lemma $5.4$]{MR3314515}  algebraic actions of the  form $\alpha_A \in \Act{\Gamma}{X_A}$ as in the statement of Theorem \ref{thm:X_A_SFT} have a dense set of homolicnic points. In particular, this proves Corollary \ref{cor:princ_alg_homoc} holds. I thank Nhan-Phu Chung for pointing out this out to me.
\end{remark}

Combining Theorem \ref{thm:priniciple_alg_pot} with Theorem \ref{thm:pot_stable} we get:
\begin{cor}\label{cor:princ_alg_stbl}
Every expansive principal algebraic $\Gamma$-action  is topologically stable.
\end{cor}

Using an algebraic characterizations of expansive algebraic actions due to Chung and Li we have:

\begin{cor}
Let $\Gamma$ be a countable group.
Suppose $\alpha \in \Alg{\Gamma}{X}$ is expansive. Then
$\alpha$ is algebraically  conjugate to a subsystem  of an algebraic action that is expansive and  satisfies p.o.t.
In other words,
there exists an expansive $\beta \in \Alg{\Gamma}{Y}$ that satisfies p.o.t and a $\Gamma$-equivariant  continuous monomorphism $\Phi:X \to Y$.
\end{cor}
\begin{proof}
Suppose $\alpha \in \Alg{\Gamma}{X}$ is expansive. As in \cite{MR3314515}, we can identify the left $\mathbb{Z}\Gamma$-module $\widehat{X}$ with $(\mathbb{Z}\Gamma)^k/J$ for some $k \in \mathbb{N}$ and some left  $\mathbb{Z}\Gamma$-submodule $J$.
By \cite[Theorem $3.1$]{MR3314515} there exists $A \in M_k(\mathbb{Z}\Gamma)$ invertible in $M_k(\ell^1(\Gamma))$ so that the rows of $A$ are contained in $J$.
Then $\alpha$ is algebraically  conjugate to a subsystem  of the natural algebraic $\Gamma$-action on the dual of  $(\mathbb{Z}\Gamma)^k/(\mathbb{Z}\Gamma)^kA$.
By Theorem \ref{thm:priniciple_alg_pot}  this $\Gamma$-action satisfies p.o.t.
\end{proof}

As we mentioned, for algebraic actions every dynamical property corresponds to some algebraic property of $\mathbb{Z}\Gamma$-modules.
Is there a natural algebraic interpretation for p.o.t in terms of the  Pontryagin dual?
Theorem \ref{thm:X_A_SFT} is a sufficient condition for algebraic actions to satisfy p.o.t. The following gives a natural necessary condition for expansive actions to satisfy p.o.t:


Given a ring $\mathcal{R}$, recall that a left $\mathcal{R}$-module $\mathcal{M}$ is \emph{finitely presented} if it is isomorphic to $\mathcal{R}^k/J$, where $J \subset \mathcal{R}^k$ is a finitely generated left ideal.

\begin{prop}\label{thm:finitely_presented_finite_type}
Let $\Gamma$ be a countable group and $\act{\Gamma}{X}$ an algebraic action.
If $\alpha \in \Alg{\Gamma}{X}$  is expansive and satisfies p.o.t, then the Pontryagin dual $\widehat X$ of $X$ is a finitely presented left $\ZZ\Gamma$-module.

\end{prop}

\begin{proof}
Let $\alpha \in \Alg{\Gamma}{X}$  be expansive. 
In particular it  follows that the dual module $\widehat{X}$ is finitely generated as $\mathbb{Z}\Gamma$-module \cite{MR1036904}.
Furthermore, by a characterization of expansive algebraic actions given in \cite{MR3314515}, there exists $k \in \mathbb{N}$, a left $\mathbb{Z}\Gamma$-submodule $J$ of $(\mathbb{Z}\Gamma)^k$, and a $A \in M_k(\ZZ\Gamma)$  invertible in $M_k(\ell^1(\Gamma))$ such that the rows of $A$ are contained in $J$
so that $\widehat{X}$ is isomorphic as a $\ZZ\Gamma$-module to $(\mathbb{Z}\Gamma)^k / J$.
Suppose $J$ is not a finitely generated $\ZZ\Gamma$-module. Let $J_1$ be the $\ZZ\Gamma$-module generated by the rows of $A$. Because $J$ is not finitely generated, there exists a strictly increasing sequence of left $\mathbb{Z}\Gamma$-modules
$$J_1 \subset J_2 \subset J_3 \subset \ldots,$$
so that the row of $A$ are contained in $J_1$ and $J = \bigcup_{n=1}^\infty J_n$. It follows that $X_A$ is expansive and that $X \in  \subG{\alpha}{X_A}$ is a strictly decreasing intersection of the systems $\widehat{(\mathbb{Z}\Gamma)^k)/J_n}$. By Proposition \ref{prop:finite_type} and Remark \ref{rem:stable_intersection}, it follows that $\alpha$ does not  satisfy p.o.t.
\end{proof}

In the case $X$ is a  disconnected compact metrizable abelian group then every expansive $\alpha \in \Alg{\Gamma}{X}$ is isomorphic to an abelian \emph{group shift}.
In this case it was shown in \cite{MR1345152} that if $\Gamma$ is polycyclic-by-finite then  $\alpha$ as above is topologically conjugate to a shift of finite type.
In fact, the arguments above and those of  \cite{MR1345152} easily imply that in the totally disconnected case an expansive algebraic action satisfies p.o.t if and only if the dual module is finitely presented.


In view of the above, one might try to guess that an expansive algebraic $\Gamma$-action satisfies p.o.t if and only if he Pontryagin  is a finitely presented left $\ZZ\Gamma$-module.

The following example shows that for some expansive algebraic actions, having a  finitely presented Pontryagin dual does not imply p.o.t:
\begin{example}
Let $\Gamma= \langle a,b \rangle$ be the free group generated by two elements $a,b$. Consider the natural algebraic $\Gamma$-action  on the  Pontryagin dual of the following finitely presented  left $\ZZ\Gamma$-module:
$$\widehat{X} = \mathbb{Z}\Gamma / \langle a-2,b-2 \rangle.$$
So
$$X = \left\{ x\in (\mathbb{R}/\mathbb{Z})^\Gamma~:~  x_{ga} = 2 x_{g} \mbox { and } x_{gb} = 2 x_{g} \mbox{ for every } g \in \Gamma \right\}.$$
 Let
$$Y =  \left\{ x\in (\mathbb{R}/\mathbb{Z})^\Gamma~:~  x_{ga} = 2 x_{g}\right\}.$$
Then $\act{\Gamma}{Y}$ is expansive. Now let
$$ X_n = \left\{ x\in (\mathbb{R}/\mathbb{Z})^\Gamma~:~  x_{ga} = 2 x_{g} \mbox { and } \rho_\infty(x_{gb},2 x_{g})  \le  \frac{1}{n} \mbox{ for every } g \in \Gamma \right\}.$$
Then $X= \bigcap_{n=1}^\infty X_n$ is a strictly decreasing intersection of expansive $\Gamma$-systems, so it does not satisfy p.o.t., by Proposition \ref{prop:finite_type} and Remark \ref{rem:stable_intersection}.
\end{example}


\begin{question}\label{ques:noetherian_pot}
Let $\Gamma$ be a  countable group and suppose $\ZZ \Gamma$ is left Noetherian (for example, suppose $\Gamma$ is polycyclic-by-finite). Does every expansive $\alpha \in \Alg{\Gamma}{X}$ satisfy p.o.t?
\end{question}

An affirmative answer to Question \ref{ques:noetherian_pot} would in particular recover \cite[Corollary 9.12]{MR3314515}, an affirmative answer to Question \ref{ques:Chung_Li} for expansive algebraic $\Gamma$-actions when $\ZZ \Gamma$ is left Noetherian.

\section{An expansive algebraic action with positive entropy and trivial homoclinic group}\label{sec:algebraic_exp_no_asymp}
In this section we prove Theorem \ref{thm:alg_exp_pos_ent_no_asymp}, providing a negative answer to Question \ref{ques:Chung_Li} within the class of algebraic actions.
By \cite{MR3314515},  if  $\alpha \in \Alg{\Gamma}{X}$ is an algebraic counterexample,  the acting group $\Gamma$ cannot be polycyclic-by-finite.
The construction makes crucial use of the fact that $\mathbb{Z}\Gamma$ is not left Noetherian for the group $\Gamma$ we use.


Let $(\Gamma_n)_{n=1}^\infty$ be a sequence of finite groups, and let $\Gamma := \bigoplus_{n=1}^\infty \Gamma_n$ be the direct sum of these groups.
Namely, $\Gamma$ is the countable subgroup of $\prod_{n=1}^\infty \Gamma_n$
\begin{equation}
\Gamma =\bigcup_{N=1}^\infty\left\{ (g_n)_{n=1}^\infty\in  \prod_{n=1}^\infty \Gamma_n ~:~ g_n=1_{\Gamma_n} \mbox{ for all } n \ge N\right\}.
\end{equation}


For every $n \in \mathbb{N}$, we naturally identify $\Gamma_n$ with the following subgroup of $\Gamma$:
$$\Gamma_n \cong \left\{ (g_k)_{k=1}^\infty  \in \Gamma ~:~ g_k = 1_{\Gamma_k} \mbox{ for all } k \ne n \right\}$$

The reader can keep in mind the case $\Gamma_n = (\mathbb{Z}/2\mathbb{Z})^{a_n}$, where $a_n$ is some rapidly increasing sequence of integers.
In this case $\Gamma$ will be isomorphic to the countable abelian $2$-torsion group $\bigoplus_{\mathbb{N}}(\mathbb{Z}/2\mathbb{Z})$, also isomorphic to the additive group of polynomials over the finite field of size $2$.

For every $n \in \mathbb{N}$ let
\begin{equation}
\tilde \Gamma_n := \left\{(g_n)_{k=1}^\infty \in \Gamma~:~ g_k =1_{\Gamma_k} \mbox{ for all } k > n \right\}.
\end{equation}
 Clearly   $\tilde \Gamma_n \cong \bigoplus_{k=1}^{n}\Gamma_k$ is a finite subgroup of $\Gamma$. 
 Also, the sequence $(\tilde \Gamma_n)_{n=1}^\infty$ is a left-F{\o}lner sequence for $\Gamma$.

For $n \ge 1$ let $f_n \in \mathbb{Z}\Gamma$ be given by
\begin{equation}
 f_n := \sum_{g \in \Gamma_n}g,
\end{equation}
Let $J$ be the left $\mathbb{Z}\Gamma$-ideal generated by the element $2 \in \mathbb{Z}\Gamma$ and by $\{f_n\}_{n=1}^\infty$,
and let
$$ X :=  \widehat{\mathbb{Z} \Gamma / J}.$$
Then $X$ is a totally disconnected compact abelian group.  We will identify $X$ with the following $\Gamma$-subshift:
$$ X = \left\{ x \in (\mathbb{Z}/2\mathbb{Z})^\Gamma ~:~ x\cdot f_n = 0 \mbox{ for every } n \in \mathbb{N} \right\}=$$
$$=\left\{ x \in (\mathbb{Z}/2\mathbb{Z})^\Gamma ~:~ \sum_{h \in \Gamma_n}x_{gh} =0 \mbox{ for every } n \in \mathbb{N} \mbox{ and } g \in \Gamma\right\}.$$
For every $n \in \mathbb{N}$, choose $\gamma_n \in \Gamma_n$ so that $\gamma_n \ne 1$ for all but finitely many $n$'s.
Let
\begin{equation}
E = \{ (g_n)_{n=1}^\infty \in \Gamma~:~ g_n \ne \gamma_n \mbox{ for every } n\ge 1\}.
\end{equation}

\begin{lem}\label{lem:X_E_ext}
For every $w \in \{0,1\}^E$ there exists $x \in X$ such that $x\mid_E = w$.
\end{lem}
\begin{proof}
For $g \in \Gamma$ let
\begin{equation}
I(g) :=  \left\{ n\in \mathbb{N}~:~ g_n =\gamma_n\right\}
\end{equation}
and
\begin{equation}
T(g):= \left\{ (h_n)_{n \in I(g)}:~ h_n \in \Gamma_{n} \setminus \{\gamma_n\}\right\}.
\end{equation}
Fix $w \in \{0,1\}^E$. We define $x \in \{0,1\}^\Gamma$  as follows:
\begin{equation}
x_g := \begin{cases}
\omega_g & g \in E\\
\sum_{(h_n)_{n \in I(g)} \in T(g)}\omega_{g\prod_{n \in I(g)}h_n} & \mbox{otherwise}.
\end{cases}
\end{equation}
Note that the definition is independent of the order of the product $\prod_{n \in I(g)}h_n$, because $\Gamma_n$ and  $\Gamma_m$ are commuting subgroups of $\Gamma$ for $n \ne m$.
It is clear that $x\mid_E = \omega$.
We will now verify that $x \in X$. Equivalently, we need to show that for every $g \in \Gamma$, and every $n \in \mathbb{N}$,
\begin{equation}\label{eq:sum_x_g}
\sum_{h \in \Gamma_n} x_{gh} = 0 \mod 2.
\end{equation}

Fix $g \in \Gamma$ and $n \in \mathbb{N}$.
We have
$$\sum_{h \in \Gamma_n} x_{gh} = \sum_{h \in \Gamma_n} x_{gg_n^{-1}h} =
 \sum_{h \in  \Gamma_n \setminus  \{\gamma_n\}} x_{gg_n^{-1}h} +  x_{g g_n^{-1} \gamma_n}.$$
The first equality follows because $g_n \in \Gamma_n$ so $g_n\Gamma_n= \Gamma_n$,  so the summands in the sum on the left hand side of the equality are a permutation of the summands on the right.

Let
$$I(g)\setminus \{n\} = \left\{n_1,\ldots,n_k \right\}.$$
Then we have
$$I(g g_n^{-1}\gamma_n) = \{n\} \uplus \left\{n_1,\ldots,n_k \right\}.$$
By definition of $x$ it this follows that
$$ x_{g g_n^{-1} \gamma_n} = \sum_{h \in \Gamma_n \setminus \{\gamma_n\}} \sum_{h_1 \in \Gamma_{n_1}\setminus \{\gamma_{n_1}\}}\ldots\sum_{h_k \in \Gamma_{n_k}\setminus \{\gamma_{n_k}\}} \omega_{gg_n^{-1}\gamma_n h h_1\ldots h_k}.$$
Similarly for $h \in \Gamma_n \setminus  \{\gamma_n\}$ we have
$$I(gg_n^{-1}h) = \left\{n_1,\ldots,n_k \right\}.$$ Thus, for every $h \in \Gamma_n \setminus  \{\gamma_n\}$ we have:
$$  x_{gg_n^{-1}h} = \sum_{h_1 \in \Gamma_{n_1}\setminus \{\gamma_{n_1}\}}\ldots\sum_{h_k \in \Gamma_{n_k}\setminus \{\gamma_{n_k}\}} \omega_{gg_n^{-1}\gamma_n h h_1\ldots h_k}.$$
Thus,
$$ x_{g g_n^{-1} \gamma_n}=  \sum_{h \in  \Gamma_n \setminus  \{\gamma_n\}} x_{gg_n^{-1}h}.$$

We have thus shown that  \eqref{eq:sum_x_g} holds. So $x \in X$ and $x\mid_E = \omega$.
\end{proof}
\begin{lem}\label{lem:X_alg_lower_bound}
\begin{equation}\label{eq:h_lower_bound}
h(\act{\Gamma}{X}) = \prod_{n=1}^\infty(1 - |\Gamma_n|^{-1})\log 2.
\end{equation}
\end{lem}
\begin{proof}
Let
$$ N_n := \# \left\{ x\mid_{\tilde \Gamma_n} ~:~ x \in X\right\}.$$
From  Lemma \ref{lem:X_E_ext}, it follows that
$$ N_n = 2^{\prod_{k=1}^n(\Gamma_k \setminus \{\gamma_k\})}.$$
Because $\{\tilde \Gamma_n\}$ is a left-F{\o}lner sequence, it follows that
$$h(\act{\Gamma}{x})  =  \lim_{n \to \infty} \frac{\log N_n}{|\tilde \Gamma_n|}.$$
Thus,
$$h(\act{\Gamma}{x}) =\lim_{n \to \infty}\frac{\prod_{k=1}^n\left|\Gamma_k \setminus \{\gamma_k\} \right|}{\prod_{k=1}^n |\Gamma_k|} \log 2 .$$
From this we immediately get  \eqref{eq:h_lower_bound}.
\end{proof}

\begin{lem}\label{lem:X_alg_no_asymp}
The action $\act{\Gamma}{X}$ has no off-diagonal asymptotic pairs.
\end{lem}
\begin{proof}
We will show that the only homoclinic point $x \in X$ is the identity element.
In other words, we will show that if $x \in X$ and $F \subset \Gamma$ is a finite set such that
$x_g = 0$ for all $g \in \Gamma \setminus F$, then $x_g =0$ for all $g \in \Gamma$.
Indeed, because $F$ is finite, there exists $n \in \mathbb{N}$ such that $g_n = 1$ for all $g \in F$. Suppose $g \in F$.
Then for every $h \in \Gamma_n \setminus \{1\}$, $gh \not\in F$.
It follows that
$$x_g = \sum_{h \in \Gamma_n} x_{gh} =0.$$\
This proves that $x_g =0$ for all $g \in F$, thus $x$ is the identity element of $X$.
\end{proof}

The above construction with  $\Gamma_n = (\mathbb{Z}/2\mathbb{Z})^{a_n}$ completes the proof 
 of Theorem \ref{thm:alg_exp_pos_ent_no_asymp}.

Using the characterization of completely positive entropy for from \cite{MR3314515}, we can further show that the construction above has completely positive entropy:

Let $\mathcal{B}_X$ denote the Borel $\sigma$-algebra on $X$, and let $\mu_X$ denote Haar measure on $X$, normalized to be a probability measure.
Then $\act{\Gamma}{(X,\mathcal{B}_X,\mu_X)}$ is a probability preserving $\Gamma$-action.

\begin{prop}
If $ \prod_{n=1}^\infty(1 - |\Gamma_n|^{-1})>0$ then the measure preserving action $\act{\Gamma}{(X,\mathcal{B}_X,\mu_X)}$ has completely positive entropy.
\end{prop}
\begin{proof}
By \cite[Corollary $8.4$]{MR3314515} we need to show that $\mathit{IE}(X)=X$. Equivalently (see \cite[Definition $2.3$]{MR3314515})  we need to show that for any pair of non-empty open sets $U_0,U_1$ with $0 \in U_0$  there exists a finite $K \subset \Gamma$  and $c,\epsilon >0$ such that for any finite set $F \subset \Gamma$ with $|K F \setminus F| < \epsilon |F|$, the pair $(U_0,U_1)$ has an independence set $F' \subset F$ with $|F'| \ge c |F|$. In our case it suffices to prove the above holds for open sets  $U_0,U_1$ of the form:
$$U_0 = [0]_{\tilde \Gamma_n}\mbox { and } U_1 = [x]_{\tilde \Gamma_n},~ x \in X\; n \in \mathbb{N}.$$
The choice of $K$ and $\epsilon$ will not be relevant for us.
Let $$c = \frac{1}{2}|\tilde \Gamma_n|^{-1} \prod_{k=n+1}^\infty (1-|\Gamma_k|^{-1}).$$

Choose any finite set $F \subset \Gamma$.
For any choice of   $\tilde \gamma_k \in \Gamma_k$ for $1 \le k \le n$, let
$$F_{\tilde \gamma_1,\ldots,\tilde \gamma_n} := \left\{(g_k)_{k=1}^\infty \in  F~:~ g_k = \gamma_n \mbox{ for all } k\le n  \right\}.$$
It follows that  the sets $F_{\tilde \gamma_1,\ldots,\tilde \gamma_n}$ are a partition of $F$ into $|\tilde \Gamma_n|$ sets, so there exists a choice of $\tilde \gamma_1,\ldots,\tilde \gamma_n$ as above so that
$|F_{\tilde \gamma_1,\ldots,\tilde \gamma_n} | \ge |\tilde \Gamma_n|^{-1} |F|$.
Next, for every $k > n$ define $F_k \subset \Gamma$ and $\gamma_k \in \Gamma_k$ by induction as follows:
To start the induction, set $F_n := F_{\tilde \gamma_1,\ldots,\tilde \gamma_n}$.  Assume $F_k$ has been defined. Choose $\gamma_{k+1} \in \Gamma_{k+1}$ so that
$$\left|  \left\{(g_j)_{j=1}^\infty \in  F_k~:~ g_{k+1}=\gamma_{k+1}\right\}\right| \le |\Gamma_{k+1}|^{-1} |F_k|,$$
and let
$$F_{k+1} :=   \left\{(g_j)_{j=1}^\infty \in  F_k~:~ g_{k+1}\ne\gamma_{k+1}\right\}.$$
Let $F' = \bigcap_{k=n+1}^\infty F_k$.
Note that  the decreasing sequence $\ldots \subset F_{k+1} \subset F_k$ stabilizes because $F$ is finite. Also note that $\gamma_n \ne 1$ for all but finitely many $n$'s.
From the construction and choice of $c$ it is clear that
$|F'| \ge c |F|$. An application of  Lemma \ref{lem:X_E_ext} with the corresponding $\gamma_k$'s so that $\gamma_k \ne \tilde \gamma_k$ for $k \le n$, gives that $F'$ is an independence set for $(U_0,U_1)$.
\end{proof}

\section{Positive entropy subshifts without off-diagonal  asymptotic pairs}\label{sec:positive_exp_no_asymp}


In this section we prove Theorem \ref{thm:exapnsive_pos_entropy_no_asymp}.
We actually prove a slightly more general result.
This extra generalization does not seem to complicate the construction much and we hope it  helps isolate the main idea.
Throughout this section $\Gamma$ will be  a countably infinite  amenable group that is \emph{residually finite}.

Recall that a group $\Gamma$ is residually finite if and only if there exist a decreasing sequence of normal subgroups so that
\begin{equation}\label{eq:Res_fin}
\ldots \Gamma_{n+1} \lhd \Gamma_n \lhd \ldots \Gamma_1 \lhd \Gamma_0 = \Gamma, \mbox{ and } \bigcap_{n=1}^\infty \Gamma_n= \{1\}.
\end{equation}
and
\begin{equation}\label{eq:Res_fin_2}
[\Gamma:\Gamma_n] < \infty \mbox{ for every } n \ge 1.
\end{equation}


Let
\begin{equation}
\AAA= \{0,1,2\} \cong \ZZ/3\ZZ.
\end{equation}
Fix a strictly decreasing sequence of finite-index normal  subgroups $(\Gamma_n)_{n=1}^\infty$ with trivial intersection as in \eqref{eq:Res_fin}. For every $n \ge 1$, let
\begin{equation}
b_n := [\Gamma:\Gamma_n] \mbox{ and } a_n := [\Gamma_{n-1}:\Gamma_{n}].
\end{equation}
For every $n \ge 1$, choose  a set of representatives  $T_n \subset \Gamma_{n-1}$ for $\Gamma_{n}$-cosets in $\Gamma_{n-1}$. Namely, the following holds:
\begin{equation}\label{eq:T_n_g_N}
\Gamma_{n-1} = \biguplus_{t \in T_n}\Gamma_{n}t=\biguplus_{t \in T_n}t\Gamma_{n}.
\end{equation}
Furthermore, assume that  $1 \in  T_n$ for every $n \ge 1$.
Note that $|a_n| > 1$ for every $n$, and so by passing to a subsequence we can make sure the sequence $(a_n)_{n=1}^\infty$ grows sufficiently fast in order to satisfy some conditions that will be stated later.

It follows that $|T_n|= [\Gamma_{n-1}:\Gamma_{n}]= a_n$.
Also, let
\begin{equation}\label{eq:E_n_def}
E_{n} := T_n\cdot \ldots \cdot T_1 = \left\{ t_n\cdot\ldots\cdot t_1~:~ t_i \in T_i ~ i =1,\ldots,n\right\}.
\end{equation}

Note that
\begin{equation}\label{eq:G_n_E_n}
 \Gamma= \biguplus_{g \in E_n} \Gamma_{n}g= \biguplus_{g \in E_n} g\Gamma_{n},
\end{equation}
and $|E_n| = [\Gamma:\Gamma_{n}] = b_{n}$.
It will be useful to define
\begin{equation}
\Gamma_0 := \Gamma,~ T_0:= \{1\}\mbox{  and } a_0:= b_0 := 1.
\end{equation}
Note that $b_n = a_n \cdot b_{n-1}$.

For $E \subset \Gamma$, $w \in \AAA^E$ and $g \in \Gamma$, let $g \cdot w \in \AAA^{gE}$ be given by
\begin{equation}
 (g \cdot w)_{ g f}:= w_f \mbox{ for } f \in E.
\end{equation}


The first step in our construction is to  inductively define  sequences $(w^{(n)})_{n=0}^\infty$ and $(s^{(n)})_{n=1}^\infty$  with $w^{(n)},s^{(n)} \in \AAA^{E_n}$ together with  sequences $A_n \subset \AAA^{E_n}$ so that $w^{(n)} \in A_n$ as follows:

To start the induction define
$w^{(0)} \in \AAA^{T_0} \cong \AAA$  by $w^{(0)}_1 := 0$, and  define $A_0 := \AAA^{T_0}$.


Suppose $n \ge 1$, that $w^{(n-1)}$ and $A_{n-1}$ have been defined.

Let
\begin{equation}
 B_n' := \left\{ w \in \AAA^{E_n} ~:~   t^{-1} \cdot  w \mid_{E_{n-1}} \in  A_{n-1} \setminus \{w^{(n-1)}\} ~ \forall t \in T_n \setminus \{1\}\right\},
\end{equation}
and
\begin{equation}
 B_n = \left\{ w \in B_n' ~:~ w \mid_{E_{n-1}} = w^{(n-1)}\right\}.
\end{equation}


For $s \in \AAA^{E_{n-1}}$ let
\begin{equation}\label{eq:C_s_n_def}
 C_{s,n} := \left\{ w \in B_n ~:~ \sum_{t \in T_n} (t^{-1} \cdot w)\mid_{E_{n-1}} = s \right\}.
\end{equation}

In \eqref{eq:C_s_n_def} above and elsewhere we sum elements of $\AAA^{E_{n-1}}$ pointwise, as $\AAA$-valued functions, recalling that $\AAA = \mathbb{Z}/3\mathbb{Z}$.
Thus $B_n = \biguplus_{s \in \AAA^{E_{n-1}}}C_{s,n}$.
Choose  $s^{(n-1)} \in \AAA^{E_{n-1}}$ so that
\begin{equation}
|C_{s^{(n-1)},n}| = \max\left\{|C_{s,n}|~:~ s \in \AAA^{E_{n-1}}\right\}.
\end{equation}

Now define
$$A_n := C_{s^{(n-1)},n},$$
and choose $w_n$ to be an arbitrary element of $A_n$.

\begin{example}
Consider the case $\Gamma = \ZZ$. Let $\Gamma_1 = 4 \ZZ$, $\Gamma_2 = 12 \ZZ$, $E_1 = T_1 = \{0,1,2,3\}$, $T_2= \{0,4,8\}$, $E_2 = \{0,1,2,\ldots,11\}$. We have
 $a_1 = 4$ and $a_2 =3$. Because $w^{(0)} = 0 $, it follows that
$A_1 = \{1,2\}^{\{0,1,2,3\}}$.
Considering elements of $A_1$ as strings of length $4$ over the alphabet $A_1$, we have
$$ C_{0,1} = \left\{ 0111, 0222\right\}, C_{1,1} = \{0121,0211,0112\}, C_{2,1}=\{0221,0122,0212\} .$$
So  at this point the procedure above allows to choose  either  $s^{(1)} =1$ or $s^{(1)}=2$. Suppose we choose $s^{(1)} =1$. Then
$$A_1= C_{1,1} = \{0121,0211,0112\},$$
and we can choose for example $w^{(1)} = 0121$.
\end{example}

We note  that if for some $n$ it happens that $|A_n| \le 2$, then $|A_{n+1}| = 1$ and $|A_{n+2}|= \emptyset$, so we will not be able to continue the construction.
The following lemma shows that this does not happen if the sequence $(a_n)_{n=1}^\infty$ grows rapidly.
\begin{lem}\label{lem:A_n_big}
If $A_n \ne \emptyset$ then
\begin{equation}\label{eq:A_n_big}
|A_{n+1}| \ge \frac{1}{3^{b_{n}}}\left( |A_n| -1\right)^{a_{n+1} -1}.
\end{equation}

\end{lem}
\begin{proof}
Because $B_{n+1} = \biguplus_{ s\in \AAA^{E_{n}}}C_{s,n+1}$, it follows that
\begin{equation}\label{eq:b_n_ineq}
|B_{n+1}| = \sum_{s \in \AAA^{E_{n}}} |C_{s,n+1}|\le 3^{b_{n}} \max\{C_{s,n+1}~:~s \in \AAA^{E_{n}}\} =  3^{b_{n}} |A_{n+1}|.
\end{equation}
It is clear that $B_{n+1}$ is in bijection with $(A_n \setminus \{w_n\})^{a_{n+1}-1}$, so
\begin{equation}\label{eq:B_n_card}
|B_{n+1}| = \left( |A_n| -1\right)^{a_{n+1} -1}.
\end{equation}
The inequality \eqref{eq:A_n_big} follows from \eqref{eq:b_n_ineq} and \eqref{eq:B_n_card}.
\end{proof}

\begin{lem}\label{lem:A_n_not_empty}
Suppose
\begin{equation}\label{eq:a_n_big_1}
a_{n+1} \ge 2b_{n} +3 \mbox{ for every } n \ge 1.
\end{equation}
 Then for every $n \ge 1$:
\begin{equation}\label{eq:A_n_LB}
|A_n| \ge 3^{b_{n-1}}+1
\end{equation}
In particular  $A_n \ne \emptyset$ for every $n \ge 0$.
\end{lem}
\begin{proof}
Proceed  by induction:
For $n=0$, $b_0=1$ and $A_0 = \AAA$ so $|A_0| = 3 > 3^0+1=3^{b_0}+1$, so \eqref{eq:A_n_LB} is true for $n=0$.
Assume by induction that \eqref{eq:A_n_LB} holds for fixed $n$.
In particular  $|A_n| \ge 3$. Thus by Lemma \ref{lem:A_n_big} and \eqref{eq:a_n_big_1}
$$ |A_{n+1}| \ge  \frac{1}{3^{b_{n}}}3^{b_{n-1}(a_{n+1}-1)} \ge 3^{-b_{n}}\cdot 3^{b_{n-1}(2b_{n} +2)} \ge 3^{(2b_{n-1}-1)b_n}\cdot 3^{2n}\ge 3^{b_{n}} +1.$$
\end{proof}

From now on assume \eqref{eq:a_n_big_1} holds, and so by Lemma \ref{lem:A_n_not_empty} $A_n \ne \emptyset$.
For $n \in \mathbb{N}$, define $R_n \subset \AAA^\Gamma$ as follows:
\begin{equation}\label{eq:R_n_def}
R_n := \left\{ x \in \AAA^\Gamma:~ (g \cdot x)\mid_{E_n} \in A_n\mbox{ for every } g \in \Gamma_n\right\}.
\end{equation}

Clearly $R_n$ is a compact $\Gamma_n$-invariant subset of $\AAA^\Gamma$. By \eqref{eq:G_n_E_n},  the action $\act{\Gamma_n}{R_n}$  is trivially isomorphic to the shift action $\act{\Gamma_n}{A_n^{\Gamma_n}}$.

\begin{lem}\label{lem:R_n_nested}
For every $n \ge 0$
$$R_{n+1} \subset R_{n}.$$
\end{lem}
\begin{proof}
Suppose $x \in R_{n+1}$.  Choose $g \in \Gamma_{n}$. By \eqref{eq:T_n_g_N}, $\Gamma_{n} =  \Gamma_{n+1}T_{n+1}$. Thus there exists $g' \in \Gamma_{n+1}$ and $t \in T_{n+1}$ so that $g= g't$.
By definition of $R_{n+1}$, $w := ((g')^{-1} \cdot x )\mid_{E_{n+1}} \in A_{n+1}$. By definition of $A_{n+1}$, it follows that $(t^{-1} \cdot w) \mid_{E_{n}} \in A_{n}$. But
$(t^{-1} \cdot w ) \mid_{E_{n}} = (g^{-1} \cdot x )\mid_{E_{n}}$, so $x \in R_{n}$.
\end{proof}

\begin{lem}\label{lem:R_n_trans_disjoint}
For every $n \ge 1$ , $x \in R_n$ we have
\begin{equation}\label{eq:R_n_trans_distjoint}
(g \cdot x )\mid_{E_n} \in A_n \mbox{ if and only if } g \in \Gamma_n.
\end{equation}
\end{lem}

\begin{proof}
We prove this by induction on $n$.
To start the induction, note that if $x \in R_1$ then $x_{g}=0$ if and only if $g \in \Gamma_1$ so $x \cdot g \in R_1 $ if and only if $g \in \Gamma_1$.

For the induction step, fix $x \in R_n$. From the definition of $R_n$ it is clear that if $g \in \Gamma_n$ then $(g \cdot x)\mid_{E_n} \in A_n$.  By Lemma \ref{lem:R_n_nested} $R_{n} \subset R_{n-1}$, so by induction hypothesis, for every $g \in \Gamma \setminus  \Gamma_{n-1}$  we have
$ (g \cdot x)\mid_{E_{n-1}} \not\in A_{n-1}$.

It remains to show that $g^{-1} \cdot x  \not \in R_{n-1}$ for every  $g \in \Gamma_{n-1} \setminus \Gamma_n$.
Choose $g \in \Gamma_{n-1} \setminus \Gamma_n$. Because $\Gamma_{n-1} =  \biguplus_{t \in T_{n-1}} \Gamma_{n} t$,  there exists  $g' \in \Gamma_n$ and $t \in T_{n-1} \setminus \{1\}$ so that $g = g't$.
Let $v : = ((g')^{-1} \cdot x)\mid_{E_n}$.
Because $x \in R_n$, it follows from the definition of $R_n$ that $ v \in A_{n}$,
so
$(t^{-1} \cdot v)\mid_{E_{n-1}} \in A_{n-1}  \setminus \{w^{(n-1)}\}$.
But  $(t^{-1} \cdot v)\mid_{E_{n-1}} = (g^{-1} \cdot x )\mid_{E_{n-1}}$, so  $g^{-1} \cdot x \not \in R_{n}$.
\end{proof}

Define
\begin{equation}\label{X_n_def}
X_n := \bigcup_{g \in E_n}g \cdot R_n.
\end{equation}

$X_n$ is a finite union of  compact subsets so it is compact.
Also, because $R_n$ is $\Gamma_n$-invariant, it follows that $X_n$ is $\Gamma$-invariant.

\begin{lem}\label{lem:X_n_nested}
For every $n \ge 1$, $X_{n+1} \subset X_n$.
\end{lem}
\begin{proof}
By Lemma \ref{lem:R_n_nested} $R_{n+1} \subset R_n$ so
$$X_{n+1} = \bigcup_{g \in E_{n+1}}g\cdot R_{n+1}   \subset \bigcup_{g \in E_{n+1}}g\cdot R_{n} $$
Because $E_{n+1}= T_{n+1} E_n $,
$$\bigcup_{g \in E_{n+1}}g \cdot R_{n}  = \bigcup_{t \in T_{n+1}} t \cdot \left(  \bigcup_{g \in E_{n}} g\cdot R_n \right)  =\bigcup_{t \in T_{n+1}} t\cdot  X_n.$$
Note that  $t \cdot X_n  = X_n$ for $t \in T_{n+1}$.
We conclude that indeed $X_{n+1} \subset X_n$.
\end{proof}


We now define:
\begin{equation}
X:= \bigcap_{n=1}^\infty X_n.
\end{equation}



\begin{lem}\label{lem:X_n_assympt}
Suppose $(x,y) \in X_{n+1} \times X_{n+1}$ and $F \subset \Gamma$ are such that
\begin{equation}\label{eq:x_y_sim}
x_g=y_g \mbox{ for all  } g \in \Gamma \setminus F.
\end{equation}
In addition suppose that
$g_1\Gamma_n \ne g_2\Gamma_n$ for every pair of distinct elements $g_1,g_2 \in F$.
Then $x=y$.
\end{lem}
\begin{proof}
Let $(x,y) \in   X_{n+1} \times X_{n+1}$ and $F \subset \Gamma$ be as above. By \eqref{X_n_def} there exists $f,\tilde f \in E_{n+1}$ so that  $f \cdot x \in  R_{n+1} $ and $\tilde f \cdot y \in  R_{n+1} $.
Because $f \cdot x \in R_{n+1}$, it follows that
$(g\cdot f \cdot x)\mid_{E_{n+1}} \in A_{n+1}$ for all $g \in \Gamma_{n+1}$.

By \eqref{eq:x_y_sim}
\begin{equation}\label{eq:x_y_t_sim}
( f \cdot x)_g=(f\cdot y )_g \mbox{ for all  } g \in \Gamma \setminus (f F)
\end{equation}
Because $\Gamma_{n+1}$ is infinite and $F$ is finite, 
by \eqref{eq:x_y_t_sim} there are infinitely many $g \in \Gamma_{n+1}$ so that 
$$ (g^{-1} \cdot f \cdot y)\mid_{E_{n+1}} = (g^{-1} \cdot f \cdot x)\mid_{E_{n+1}} \in A_{n+1}.$$
By Lemma \ref{lem:R_n_trans_disjoint} it follows that 
$f \cdot y\in  R_{n+1}$.

Because $f \cdot x,f \cdot y \in R_{n+1}$ it follows that for every $g \in \Gamma_{n+1}$,
 $$(g^{-1} \cdot f \cdot x)\mid_{E_{n+1}},(g^{-1} \cdot f \cdot y)\mid_{E_{n+1}} \in A_{n+1}.$$
Thus for every $r \in E_{n+1}$
\begin{equation}\label{eq:sum_eq_u_v}
 \sum_{t \in T_{n+1}} (t^{-1} \cdot g^{-1} \cdot f \cdot x)_r =\sum_{ t \in T_{n+1}}  (t^{-1} \cdot g^{-1} \cdot f \cdot y)_r = s^{(n)}_r.
\end{equation}
Equivalently,
\begin{equation}\label{eq:sum_eq_u_v}
 \sum_{t \in T_{n+1}} x_{f^{-1} \cdot g \cdot t \cdot r} =\sum_{ t \in T_{n+1}}  y_{f^{-1} \cdot g \cdot t \cdot r} = s^{(n)}_r.
\end{equation}

We claim that there exists at most one $t \in T_{n+1}$ so that $f^{-1} \cdot g \cdot t \cdot r \in F$. Indeed, because $T_{n+1} \subset \Gamma_n$, and because $\Gamma_n$ is normal, the set
$\left\{ f^{-1} \cdot g \cdot t \cdot r :~  t \in T_{n+1} \right\}$ is contained in a single $\Gamma_n$-coset.

It follows that there exists at most one $t \in T_{n+1}$ so that   $x_{f^{-1} \cdot g \cdot t \cdot r} \ne y_{f^{-1} \cdot g \cdot t \cdot r}$.
By \eqref{eq:sum_eq_u_v}, it follows that in fact
\begin{equation}\label{eq:u_eq_v}
x_{f^{-1}\cdot g\cdot   t \cdot  r} = y_{f^{-1} \cdot g \cdot t\cdot r} \mbox{ for every } g\in \Gamma_{n+1}~,~ t \in T_{n+1} \mbox{ and } r \in E_{n},
\end{equation}
 Every $\tilde g \in \Gamma$ can be written as $\tilde g = f^{-1} g t r$ for some $r \in E_{n+1}$, $t \in T_{n+1}$ and $g \in \Gamma_{n+1}$,
so $x_{\tilde g} = y_{\tilde g}$ for every $\tilde g \in \Gamma$.
\end{proof}

Our next goal is to  show that if the sequence $(a_n)_{n=1}^\infty$ grows sufficiently fast, then $h(X)>0$.

To be more specific:

\begin{lem}\label{lem:h_X_n_lower_bound}
If $(a_n)_{n=1}^\infty$ satisfies the assumption in the statement of Lemma \ref{lem:A_n_not_empty} then
\begin{equation}\label{h_X_n_lower_bound}
 h(\act{\Gamma}{X_n})\ge \frac{a_1-1}{a_1}\log(2)- \frac{1}{a_1}\log(3) - \sum_{k=2}^n \left(\frac{2}{3^{b_{k-2}}+1}+\frac{2}{a_k}\right)\cdot \log(3).
\end{equation}
\end{lem}
\begin{proof}
Because $\act{\Gamma_n}{tR_n}$  is isomorphic to the full-shift over the alphabet $A_n$ for every $n \ge1 $ and every $t \in E_n$,
 it follows 
that
$$h(\act{\Gamma_n}{tR_n}) = \log |A_n|.$$
Because $X_n=  \bigcup_{t \in E_n}t R_n$ is a finite union of $\Gamma_n$-invariant sets it follows that
$$h(\act{\Gamma_n}{X_n}) = \max_{t \in E_n} h(\act{\Gamma_n}{t R_n}) = \log |A_n|.$$
Thus, by the subgroup formula for entropy  (see \cite[Theorem $2.16$]{MR1878075} for a stronger result):
$$h(\act{\Gamma}{X_n})   = [\Gamma:\Gamma_n]^{-1} \log |A_n|= b_n^{-1} \log |A_n|$$

Taking logs in  \eqref{eq:A_n_big} we get:
\begin{equation}\label{eq:log_A_n}
\log |A_n| \ge (a_{n} -1) \log\left( |A_{n-1}|- 1\right) - b_{n-1} \log 3 .
\end{equation}
Denote
\begin{equation}
h_n := h(\act{\Gamma}{X_n})= b_n^{-1} \log |A_n| .
\end{equation}
Recall that $|A_0|=3$ and $b_1=a_1$. So substituting  $n=1$ in \eqref{eq:log_A_n} we get:
\begin{equation}\label{eq_h1}
h_1 = \frac{1}{a_1}\log(|A_1|) \ge \frac{a_1-1}{a_1}\log(2)- \frac{1}{a_1}\log(3).
\end{equation}
Dividing \eqref{eq:log_A_n} by $b_n$ we get:
$$\frac{1}{b_n}\log |A_n| \ge \frac{(a_n -1)\cdot b_{n-1}}{b_n} \cdot \frac{\log \left(|A_{n-1}|-1 \right)}{b_{n-1}}  -\frac{ b_{n-1}}{b_n} \log 3.$$
Using the relation $b_n = a_n b_{n-1}$ it follows that:
$$h_n \ge \left(1-  \frac{1}{a_n}\right)\left[h_{n-1}\cdot\frac{\log(|A_{n-1}|-1)}{\log |A_{n-1}|}\right] - \frac{1}{a_n}\log 3 .$$
Apply  the estimate:
$$\frac{\log(x-1)}{\log(x)} = 1 -\frac{\log(x)-\log(x-1)}{\log(x)} \ge 1- \frac{1}{(x-1)\log(x)} \ge 1- \frac{2}{x} \mbox{ for } x\ge 2,$$
with $x= |A_{n-1}|$ to get:
$$h_n\ge \left(1-  \frac{1}{a_n}\right)\left[h_{n-1}\left(1-\frac{2}{|A_{n-1}|}\right)\right] - \frac{1}{a_n}\log 3 .$$
Equivalently:
$$h_n - h_{n-1} \ge -\left(\frac{1}{a_n}+\frac{2}{|A_{n-1}|} -\frac{1}{a_n|A_{n-1}|}\right)h_{n-1}- \frac{1}{a_n}\log 3.$$
By  Lemma \ref{lem:A_n_not_empty}, for  every $n \ge 2$:
$$|A_{n-1}| \ge 3^{b_{n-2}}+1.$$
Note that $X_0= (\ZZ/3\ZZ)^\Gamma$ so $h_0= h(\act{\Gamma}{X_0}) = \log 3$. Because the sequence $\{h_n\}_{n=1}^\infty$ is monotone non-increasing $h_n \le \log 3$ for every $n \ge 1$.
We thus have:
\begin{equation}\label{eq:h_diff}
h_k - h_{k-1} \ge  -\left(\frac{2}{3^{b_{k-2}}+1}+\frac{2}{a_k}\right)\cdot \log(3) \mbox{ for every } k \ge 2.
\end{equation}
and
Summing \eqref{eq:h_diff} over $2 \le k \le n$ we get
$$ h_n \ge h_1 - \sum_{k=2}^n \left(\frac{2}{3^{b_{k-2}}+1}+\frac{2}{a_k}\right)\cdot \log(3).$$
The estimate \eqref{h_X_n_lower_bound} follows using  the estimate \eqref{eq_h1} for $h_1$.
\end{proof}

To conclude our construction:
\begin{prop}
If 
 $(a_n)_{n=1}^\infty$ satisfies the assumption in the statement of Lemma \ref{lem:A_n_not_empty}  and $a_1,a_2,\ldots $ are sufficiently big then
 $h(\act{\Gamma}{X})>0 $  and $X$ has no off-diagonal asymptotic pairs.
\end{prop}
\begin{proof}
Because $X = \bigcap_{n=1}^\infty X_n$, and $X_{n+1} \subset X_n$ it follows from  
upper semi-continuity of topological entropy (see \cite[Appendix $A$]{ETS:10144219})
that
$$h(\act{\Gamma}{X})= \inf_n  h(\act{\Gamma}{X_n}).$$
By  Lemma \ref{lem:h_X_n_lower_bound},
$$\inf_n h(\act{\Gamma}{X_n}) \ge \frac{a_1-1}{a_1}\log(2)- \frac{1}{a_1}\log(3) - \sum_{k=2}^\infty \left(\frac{2}{3^{b_{k-1}}+1}+\frac{2}{a_k}\right)\cdot \log(3).$$
If $a_1$ is sufficiently big than $ \frac{a_1-1}{a_1}\log(2)- \frac{1}{a_1}\log(3)  > \log(2) -\epsilon$. Also, by choosing $a_2,a_3,\ldots$ to grow sufficiently fast
the series $\sum_{k=2}^\infty \left(\frac{2}{3^{b_{k-2}}+1}+\frac{2}{a_k}\right)$ converges and the sum can be made smaller than $\frac{1}{2}\log(3) + \epsilon$ for any $\epsilon >0$.
This proves $\act{\Gamma}{X}$ has  positive entropy as soon as $a_1,a_2,\ldots $ are sufficiently big.

Let us show that $X$ has no  off-diagonal asymptotic pairs.
Suppose $(x,y) \in X\times X$ is an asymptotic pair.   If follows
that $(x,y)$  satisfy \eqref{eq:x_y_sim} for some finite $F \subset \Gamma$.
Because $\bigcap_n \Gamma_n = \{1\}$ it follows that there exists $n$ so that the map $g \mapsto g\Gamma_n$ is injective of $F$.
 Because $(x,y) \in X \times X \subset X_{n+1} \times X_{n+1}$ it follows from  Lemma \ref{lem:X_n_assympt} that $x=y$.
\end{proof}

\bibliographystyle{abbrv}
\bibliography{asymptotic_pairs}
\end{document}